\documentclass[12pt]{amsart}
\usepackage{amsmath,amsfonts,amsthm,amscd,amssymb,mathrsfs,amssymb}
\usepackage{stmaryrd}
\usepackage{graphicx}
\newtheorem{theorem}{Theorem}

\newtheorem{definition}[theorem]{Definition}

\newtheorem{corollary}[theorem]{Corollary}

\newtheorem{remark}[theorem]{Remark}

\newcommand{\CP}{\mathbb{CP}}
\newcommand{\CC}{\mathbb{C}}

\newcommand{\RR}{\mathbb{R}}

\newcommand{\z}{\mathbf{z}}
\newcommand{\x}{{\boldsymbol\xi}}

\newcommand{\PP}{\mathbb{P}}
\numberwithin{equation}{section}
\numberwithin{theorem}{section}
\numberwithin{table}{section}
\numberwithin{table}{section}
\begin{document}
\bibliographystyle{amsalpha} 

\title[K\"ahler conformal compactifications]{On K\"ahler conformal compactifications of $U(n)$-invariant ALE spaces}
\author{Michael G. Dabkowski}
\address{Department of Mathematics, University of Michigan, 
MI, 48109}
\email{mgdabkow@umich.edu}
\author{Michael T. Lock}\footnote[1]{Research supported in part by the NSF RTG Grant DMS-1148490.}
\address{Department of Mathematics, The University of Texas at Austin, 
TX, 78712}
\email{mlock@math.utexas.edu}
\begin{abstract}
We prove that a certain class of ALE spaces always has a K\"ahler conformal compactification, and moreover provide explicit formulas for the conformal factor and the K\"ahler potential of said compactification.  We then apply this to give a new and simple construction of the canonical Bochner-K\"ahler metric on certain weighted projective spaces, and also to explicitly construct a family K\"ahler edge-cone metrics on $\CP^2$, with singular set $\CP^1$, having cone angles $2\pi\beta$ for all~$\beta>0$.  We conclude by discussing how these results can be used to obtain certain well-known Einstein metrics.
\end{abstract}
\maketitle
\vspace{-.3in}

\section{Introduction}
This work is motivated by the conformal relationship between two well-known K\"ahler metrics.  The Burns metric on $\CC^2$ blown-up at the origin is conformal in an orientation reversing manner to the Fubini-Study metric on $\CP^2$ minus a point.  In fact, it extends smoothly to the one-point compactification, see \cite{LeBrunnegative}.  While there is a conformal relationship between these metrics, it is important to note that they are 
K\"ahler with respect to different complex structures.  There are a wide variety of K\"ahler metrics on noncompact spaces, and we investigate when relationships such as this can occur in a more general situation.

We consider K\"ahler metrics on noncompact manifolds which asymptotically look like
$\RR^n/\Gamma$, for a finite subgroup $\Gamma\subset {\rm{SO}}(n)$ acting freely on $\RR^n\setminus\{0\}$, with the metric induced from the Euclidean metric.  More precisely, we have the following definition.
\begin{definition}
\label{ALEdef}
{\em 
We say that a complete Riemannian manifold $(X,g)$ is {\em asymptotically locally Euclidean\em} (ALE) of order $\tau$ if there exists a finite subgroup $\Gamma \subset {\rm{SO}}(n)$ which acts freely on $\RR^n\setminus \{0\}$, a compact subset $K\subset X$, and a diffeomorphism
\begin{align*}
\Phi:X\setminus K\rightarrow (\RR^n\setminus B(0,R))/\Gamma,
\end{align*}
satisfying
\begin{align*}
(\Phi_*g)_{ij}&=\delta_{ij}+\mathcal{O}(r^{-\tau})\\
\partial^{|k|}(\Phi_*g)_{ij}&=\mathcal{O}(r^{-\tau-k})
\end{align*}
for any partial derivative of order $k$ as $r\rightarrow \infty$, where $r$ is the distance to a fixed basepoint.  We call $\Gamma$ the group at infinity.\em}
\end{definition}

Since ALE spaces have a group action at infinity, we should not expect them to compactify to a smooth manifold.  Instead, we look for compactifications to orbifolds with isolated singularities modeled on $\RR^n/ \Gamma$ for some $\Gamma$ as above.  We say that $g$ is a  {\em Riemannian orbifold metric with isolated singularities\em} on an $n$-manifold $M$ if it is a smooth Riemannian metric away from finitely many singular points, and at  any singular point the metric is locally the quotient of a smooth $\Gamma$-invariant metric on $B^n$ by the orbifold group $\Gamma$.

We will also be interested in another class of singular metrics which is a generalization of an orbifold metric having a higher dimensional singular set.   Let $M$ be a smooth $n$-manifold with a smoothly embedded $(n-2)$-dimensional submanifold $\Sigma$.   Near any point $p\in\Sigma$ choose coordinates $(y_1,y_2,x_1,\dots,x_{n-2})$ so that $\Sigma$ is given by $y_1=y_2=0$, and then change coordinates to a transversal polar coordinate system by setting $y_1=r\cos\theta$ and $y_2=r\sin\theta$.  We say that $g$ is a {\em Riemannian edge-cone metric\em} on $(M,\Sigma)$ with cone angle $2\pi\beta$ if it is a smooth Riemannian metric on $M\setminus \Sigma$, and around any $p\in \Sigma$ the metric can be expressed as
\begin{align}
g=dr^2+\beta^2r^2(d\theta+u_idx^i)^2+f_{ij}dx^i\otimes dx^j+r^{1+\epsilon}h,
\end{align}
where the $f_{ij}$, which are symmetric in $i$ and $j$, and the $u_i$ are smooth functions on $\Sigma$, and $h$ is a symmetric two-tensor field with infinite conormal regularity along $\Sigma$.  See \cite{JMR, AtiyahLeBrun} for more details.

\begin{remark}
{\em The definition of ALE space extends naturally to include Riemannian orbifold and edge-cone metrics.  We will distinguish by describing the space as nonsingular or by the type of singularity as needed.\em}
\end{remark}

\subsection{K\"ahler conformal compactifications}
\label{results}
Given a K\"ahler ALE space, we want to conformally relate it to 
a compact K\"ahler manifold, orbifold, or edge-cone space.
A {\em{conformal compactification}} of an ALE space $(X,g)$,
is a choice of a conformal factor $u : X \rightarrow \RR^+$ 
such that $u =\mathcal{O}(r^{-2})$ as $r \rightarrow \infty$, and we denote the compactified space by $(\widehat{X},\widehat{g})$, where $\widehat{X}=X\cup \{\infty\}$ is a compact orbifold and $\widehat{g}=u^2g$.
If $(X,g)$ is an
ALE space with group at infinity $\Gamma$, then the conformal compactification naturally reverses orientation and has orbifold group $\Gamma$.  However, if we reverse the orientation of $(\widehat{X}, \widehat{g})$ the orbifold group will be the be the orientation reversed conjugate group
$\overline{\Gamma}$.  By this we mean that the action is conjugate to that of $\Gamma$ by an element of
${\rm{O}}(n)\setminus {\rm{SO}}(n)$.

The focus here will be on {\em ${\rm{U}}(n)$-invariant K\"ahler ALE spaces\em}.  These are K\"ahler manifolds $(X,\mathcal{J},g)$ whose underlying Riemannian manifolds are ALE spaces, and whose metrics satisfy the rotational symmetry condition that they arise from a smooth potential function $\phi(\z)$ on $\CC^n\setminus \{0\}$, where $\z=|z_1|^2+\cdots +|z_n|^2$.  The actual K\"ahler ALE space can then be obtained by taking the appropriate quotient of $\CC^n\setminus\{0\}$ by
$\mathbb{Z}_k$, given by $(z_1,\cdots, z_n)\mapsto(e^{2\pi i/k}z_1,\cdots,e^{2\pi i/k} z_n)$, and attaching a $\CP^{n-1}$ or smooth point at the origin so that the metric is complete.  Note that this has a clear extension to ${\rm{U}}(n)$-invariant K\"ahler ALE spaces with edge-cone singularities along the $\CP^{n-1}$, or orbifold points, at the origin.  The underlying manifold of such ALE spaces obtained by  attaching a $\CP^{n-1}$ at the origin is the total space of $\mathcal{O}_{\PP^{n-1}}(-k)$, the $k^{th}$-power of the the tautological line bundle over $\CP^{n-1}$.

We will show that every ${\rm{U}}(n)$-invariant K\"ahler ALE metric on $\mathcal{O}_{\PP^{n-1}}(-k)$, including those with edge-cone singularities, admits a conformal compactification to a metric that is K\"ahler with respect to a {\em different complex structure}.  
From here on, this will be referred to as a {\em K\"ahler conformal compactification\em}.
This idea of a conformal structure containing two representatives which are K\"ahler with respect to different complex structures was examined in real $4$-dimensions in \cite{Apostolov-Calderbank-Gauduchon}.  In arbitrary real even-dimensions, the cohomogeneity-one situation, which is our focus, has certain structural properties that we are able to exploit in order to obtain interesting results.  We now state our main theorem.

\begin{theorem}
\label{maintheorem}
Let $g$ be a ${\rm{U}}(n)$-invariant K\"ahler ALE metric on $\mathcal{O}_{\PP^{n-1}}(-k)$, which arises from the K\"ahler potential $\phi(\z)$.  Then, there exists a  ${\rm{U}}(n)$-invariant K\"ahler conformal compactification given by the conformal factor 
\begin{align*}
u^2=\Big(\z\frac{\partial \phi}{\partial \z}\Big)^{-2}.
\end{align*}
Moreover, this is the unique such conformal factor up to scale.
\end{theorem}

Topologically, the conformal compactification of $\mathcal{O}_{\PP^{n-1}}(-k)$ is the weighted projective space $\CP^n_{(1,\cdots,1,k)}$, see Section \ref{weighted}.  The proof of Theorem \ref{maintheorem} is given in Section~\ref{main}.  In addition to finding the conformal factor, which gives a K\"ahler conformal compactification of any ${\rm{U}}(n)$-invariant K\"ahler ALE space, we are actually able to modify the proof to be able to find the K\"ahler potential of the compactification.  We state this result as the following corollary, which is also proved in Section \ref{main}.

\begin{corollary}
\label{potentialcorollary}
Let $g$ be a ${\rm{U}}(n)$-invariant K\"ahler ALE metric on $\mathcal{O}_{\PP^{n-1}}(-k)$, which arises from the K\"ahler potential $\phi(\z)$.   Then the corresponding K\"ahler
conformal compactification, of Theorem \ref{maintheorem}, has K\"ahler potential 
\begin{align*}
\widehat{\phi}(\z)=-\int \Big(\z^2\frac{\partial \phi}{\partial\z}\Big)^{-1}d\z.
\end{align*}
\end{corollary}

\begin{remark}
{\em
Both Theorem \ref{maintheorem} and Corollary \ref{potentialcorollary} extend to the situation where the metric has an edge-cone singularity.
}
\end{remark}

\begin{remark}
{\em
Consider, instead, a ${\rm{U}}(n)$-invariant K\"ahler ALE space arising from the K\"ahler potential $\phi(\z)$ and obtained by attaching a smooth point or orbifold point at the origin.  Applying the conformal factor $u^2=\big(\z\frac{\partial\phi}{\partial\z}\big)^{-2}$ to this space, yields a metric that is K\"ahler with respect to a different complex structure as well.  However, the space upon which this new metric lives would be the original space ``flipped inside out'' as this conformal factor compactifies at infinity, but blows up at the origin.
}
\end{remark}

When dealing with compactifications in general, a delicate issue arises with the regularity of the metric 
$\widehat{g}$ because a priori it is only a $C^{1,\alpha}$-orbifold metric.  However, with certain geometric conditions, one does obtain a $C^{\infty}$-orbifold metric.
For example, in dimension four, if the ALE space is anti-self-dual then the conformal compactification is a $C^{\infty}$-orbifold metric \cite{TV2, CLW}.  This was subsequently generalized to Bach-flat metrics in
\cite{Streets}, and to obstruction-flat metrics in \cite{AV12}.  

In our case of ${\rm{U}}(n)$-invariant ALE spaces, if the conformal metric $\widehat{g}=u^2g$ admits a Taylor expansion around the orbifold point, then it will be a $C^{\infty}$-orbifold metric because it would extend smoothly, in the orbifold sense, over the singular point since the expansion would be in terms of the variable $\x=|\xi_1|^2+\cdots+|\xi_n|^2$.  This amounts to checking that $\big(\frac{\partial \phi}{\partial \z}\big)^{-1}$ has a Taylor expansion at ALE infinity.
It is also likely that a ${\rm{U}}(n)$-invariant ALE space being K\"ahler scalar-flat is enough to ensure a $C^{\infty}$-orbifold metric in the compactification.  This would be interesting to investigate in light of possible applications of Theorem~\ref{maintheorem} to a family of K\"ahler scalar-flat ALE metrics on complex line bundles over $\CP^n$ which are higher dimensional generalizations of the metrics that we introduce in Section \ref{LeBrunmetrics}, see \cite{Simanca,SSChen}.  However, all of our applications of Theorem \ref{maintheorem} are to anti-self-dual metrics, so the compactifications here are already guaranteed to be $C^{\infty}$-orbifold metrics.

\subsection{Outline of paper}
K\"ahler ALE metrics play an important role in modern geometry, and it is an interesting question to understand when they admit K\"ahler conformal compactifications.  Theorem \ref{maintheorem} answers this question in the ${\rm{U}}(n)$-invariant case and Corollary \ref{potentialcorollary} actually provides an explicit formula for the K\"ahler potential of the compact orbifold metric in terms of the K\"ahler potential of the ALE space.  We prove these results in Section \ref{main}.  Next, in  Section \ref{LeBrunmetrics}, we introduce a family of K\"ahler scalar-flat ALE spaces, due to LeBrun, towards which the remainder of this work is focused.  Finally, in Section~\ref{compactifications} we provide a new and simple construction of the canonical Bochner-K\"ahler metric on a subclass of weighted projective spaces, as well as explicitly construct a family of K\"ahler edge-cone metrics on $(\CP^2,\CP^1)$.

\subsection{Acknowledgements}
The authors would like to thank Jeff A. Viaclovsky for suggesting this problem as well as for his advice and many helpful suggestions.

\section{Proofs of Theorem \ref{maintheorem} and Corollary \ref{potentialcorollary}}
\label{main}

\subsection{Proof of Theorem \ref{maintheorem}}
Let $g$ be a ${\rm{U}}(n)$-invariant K\"ahler ALE metric on $\mathcal{O}_{\PP^{n-1}}(-k)$ that arises from a K\"ahler potential $\phi(\z)$ on $\CC^n\setminus\{0\}$ with the standard complex structure $\mathcal{J}_{\mathbf{z}}$, where $\z=|z_1|^2+\cdots+|z_n|^2$.
On a spherical shell centered at the origin, where $\partial_{\z}$ and  $\bar{\partial}_{\z}$ are with respect to $\mathcal{J}_{\mathbf{z}}$, consider the K\"ahler form
\begin{align}
\omega=\sqrt{-1}\partial_{\z} \bar{\partial}_{\z}\phi(\z).
\end{align}
Because of the ${\rm{U}}(n)$-invariance, we can restrict our work to the $z_1$-axis where
\begin{align}
\label{omega_g}
\omega=\sqrt{-1}\Big[\Big(\frac{\partial \phi}{\partial \z}+\z\frac{\partial^2 \phi}{\partial \z^2}\Big)dz_1\wedge d\bar{z}_1+\frac{\partial \phi}{\partial \z}\sum_{i=2}^ndz_i\wedge d\bar{z}_i\Big].
\end{align}

We search for a conformal compactification factor $u(\z)$, depending only on the the radial variable $\z$, so that the conformal compactification $(\widehat{X},\widehat{g})$, where $\widehat{g}=u^2g$, is K\"ahler.  Since $u$ is rotationally symmetric, the metric $\widehat{g}$ would exhibit this symmetry as well around the orbifold point.  Therefore, in some coordinate system $(\xi_1,\cdots,\xi_2)$ on $\CC^n$, with the standard complex structure $\mathcal{J}_{\x}$ and where $\{0\}$ corresponds to the orbifold point, we seek to find a K\"ahler potential $\widehat{\phi}(\x)$, a function of the radial variable $\x~=~|\xi_1|^2+\cdots+|\xi_n|^2$, of some unknown conformal metric $\widehat{g}=u^2g$.  If such coordinates and potential function exist, then letting $\partial_{\x}$ and  $\bar{\partial}_{\x}$ be with respect to $\mathcal{J}_{\x}$, we could similarly examine the restriction of the K\"ahler form
\begin{align}
\widehat{\omega}=\sqrt{-1}\partial_{\x} \bar{\partial}_{\x}\widehat{\phi}(\x)
\end{align}
to the the $\xi_1$-axis, which would be
\begin{align}
\label{omega_hatg}
\widehat{\omega}=\sqrt{-1}\Big[\Big(\frac{\partial \widehat{\phi}}{\partial \x}+\x\frac{\partial^2 \widehat{\phi}}{\partial \x^2}\Big)d\xi_1\wedge d\bar{\xi}_1+\frac{\partial \widehat{\phi}}{\partial \x}\sum_{i=2}^nd\xi_i\wedge d\bar{\xi}_i\Big].
\end{align}

An inversion map would be necessary to relate the respective coordinate systems.  One initially thinks to search for a map that is holomorphic on $\CC^n\setminus \{0\}$, in which case the pullback of $\widehat{\omega}$ would be the K\"ahler form in $(z_1,\cdots,z_2)$-coordinates, and then hope that this relationship could be be used to solve for $u(\z)$.  However, there does not exist an inversion map that is holomorphic on all of $\CC^n\setminus \{0\}$ when $n>1$.  
Instead, to relate the two coordinate systems, let us choose the map $\varphi:\CC^n\setminus \{0\}\rightarrow\CC^n\setminus \{0\}$ defined by
\begin{align}
\varphi(z_1,z_2,\cdots z_n)=\Big(\frac{\bar{z}_1}{\z},\frac{z_2}{\z},\cdots,\frac{z_n}{\z} \Big)=(\xi_1,\xi_2,\cdots,\xi_n).
\end{align}
There is actually some flexibility in the choice of the inversion map, however we will see that this choice in particular will greatly simplify our work. 

It is easy to check that $\varphi^2=Id$ and that $\varphi$
is holomorphic on the $z_1/\xi_1$-axes away from the origin.  Therefore, on these axes
\begin{align}
\label{Jcommute}
\mathcal{J}_{\mathbf{\xi}}\circ \varphi_*(x)=\varphi_*\circ\mathcal{J}_{\mathbf{\z}}(x),
\end{align}
where $x\in T\RR^{2n}$ and $\CC^n$ is identified with $\RR^{2n}$ as usual.  Also, the pullback operator satisfies the usual commutativity properties with the $\partial$ and $\bar{\partial}$ operators on the $z_1/\xi_1$-axes.  However, this is not true away from these axes and even more importantly, since $\varphi$ is not holomorphic on all of $\CC^n\setminus \{0\}$, the pullback operator does not commute with the $\partial\bar{\partial}$ operator anywhere, even on the $z_1/\xi_1$-axes, because there are nonvanishing terms when one takes two derivatives.
Therefore we cannot globally understand $\widehat{\omega}$ in $(z_1,\cdots, z_n)$-coordinates directly in terms of the pullback, which is reasonable because our choice of the inversion map was somewhat arbitrary in that we do not really consider what the complex structure $\mathcal{J}_{\mathbf{\x}}$ is in $(z_1,\cdots, z_n)$-coordinates.  
Although this seems like a problem, we will now see that our choice of $\varphi$, because it is holomorphic just on the $z_1/\xi_1$-axes, will be enough for us to find $u$ and $\widehat{\phi}$ in terms of the known $\phi$.

In the respective coordinate systems, write the K\"ahler forms in terms of the metrics and complex structures as
\begin{align}
\begin{split}
\omega(x,y)=g(\mathcal{J}_{z}x, y)\phantom{==}\text{and}\phantom{==}\widehat{\omega}(v,w)={\widehat{g}}(\mathcal{J}_{\x}v, w),
\end{split}
\end{align}
where $x,y\in T\RR^{2n}$ and $v,w\in T\RR^{2n}$, and $\CC^n$ is identified with $\RR^{2n}$ as usual in the respective coordinates systems.
Even though the pullback of $\widehat{\omega}$ will not globally be the K\"ahler form, we will be able to see a useful relationship between $\varphi^*\widehat{\omega}$ and $\omega$.  Since $\widehat{g}=u^2g$, and $\varphi^*$ commutes with the complex structure on the $z_1/\xi_1$-axes as in \eqref{Jcommute}, we find the following relationship on the $z_1$-axis:
\begin{align}
\begin{split}
\label{pullback}
\varphi^*\widehat{\omega}(x,y)&=\varphi^*\widehat{g}\big(\mathcal{J}_{\x}\cdot, \cdot\big)(x,y)
=\widehat{g}\big(\mathcal{J}_{\x}\circ\varphi_*(x), \varphi_*(y)\big)\\
&=\widehat{g}\big(\varphi_*\circ \mathcal{J}_{\z}(x), \varphi_*(y)\big)
=\varphi^*\widehat{g}\big(\cdot,\cdot\big)\big(\mathcal{J}_{\z}(x), y\big)\\
&=u^2g\big(\mathcal{J}_{\z}(x), y\big)=u^2\omega(x,y).
\end{split}
\end{align}

The restriction to the $z_1$-axis has greatly simplified our work, and we
use this relationship here to set up a system of ODEs, which we then solve to find the conformal factor in terms of $\phi$ which is known.  The pullback of $\widehat{\omega}$ on the $z_1$-axis is
\begin{align}
\label{pullback_omega}
\varphi^*\widehat{\omega}=\sqrt{-1}\Big[\Big(\frac{\partial\widehat{\phi}}{\partial \z}+\z\frac{\partial^2 \widehat{\phi}}{\partial \z^2}\Big)dz_1\wedge d\bar{z}_1-\frac{\partial \widehat{\phi}}{\partial \z}\sum_{i=2}^ndz_i\wedge d\bar{z}_i\Big].
\end{align}
It is essential to notice the minus sign appearing before the second term.  
Examining the equation $\varphi^*\widehat{\omega}=u^2\omega$ on the $z_1$-axis, from \eqref{pullback}, using the formulas \eqref{pullback_omega} and \eqref{omega_g} for $\varphi^*\widehat{\omega}$ and $\omega$ on this axis respectively, we arrive at the following system of equations
\begin{align}
\begin{cases}
\label{equations}
\displaystyle u^2 \Big(\frac{\partial \phi}{\partial \z}+\z\frac{\partial ^2\phi}{\partial \z^2}\Big)= \Big(\frac{\partial \widehat{\phi}}{\partial \z}+\z\frac{\partial ^2\widehat{\phi}}{\partial \z^2}\Big)
\vspace{.2in}\\
\displaystyle u^2\Big(\frac{\partial \phi}{\partial \z}\Big)=-\frac{\partial \widehat{\phi}}{\partial \z}.
\end{cases}
\end{align}
By taking a derivative of the second equation in \eqref{equations} we see that
\begin{align}
\frac{\partial ^2\widehat{\phi}}{\partial \z^2}=-2u\frac{\partial u}{\partial \z}\frac{\partial \phi}{\partial \z}-u^2\frac{\partial ^2\phi}{\partial \z^2}.
\end{align}
Substituting this along with the second equation in \eqref{equations} into the right hand side of the first equation from \eqref{equations}, we find that 
\begin{align}
u^2 \Big(\frac{\partial \phi}{\partial \z}+\z \frac{\partial ^2\phi}{\partial \z^2}\Big )=-u^2\frac{\partial \phi}{\partial \z}-2u\z\frac{\partial u}{\partial \z}\frac{\partial \phi}{\partial \z}- u^2\z\frac{\partial ^2\phi}{\partial \z^2},
\end{align}
which simplifies to
\begin{align}
\frac{\partial}{\partial\z}\log(u)=-\frac{\partial}{\partial\z}\log\Big(\z\frac{\partial\phi}{\partial\z}\Big).
\end{align}
Hence, the conformal factor is
\begin{align}
\label{conformalfactor}
u^2=\Big(\z\frac{\partial \phi}{\partial\z}\Big)^{-2}.
\end{align}

Now, we examine the behavior of the conformal factor along the $\CP^{n-1}$ at the origin.  Since the spherical metric on $S^{2n-1}$ decomposes as $g_{S^{2n-1}}=g_{\CP^{n-1}}+h$, where $g_{\CP^{n-1}}$ is the Fubini-Study metric on $\CP^{n-1}$ and $h$ is the metric along the Hopf fiber,  the ${\rm{U}}(n)$-invariant K\"ahler ALE metric on $\mathcal{O}_{\PP^{n-1}}(-k)$ descends via the $\mathbb{Z}_k$ quotient from
\begin{align}
g=\Big[\frac{\partial}{\partial\z}\Big(\z\frac{\partial\phi}{\partial\z}\Big)\Big]\Big(d\sqrt{\z}+\z h\Big)+\Big(\z\frac{\partial\phi}{\partial\z}\Big)g_{\CP^{n-1}}.
\end{align}
Since this metric is defined along the $\CP^{n-1}$ at the origin, in other words when $\z=0$, we see that the conformal factor \eqref{conformalfactor} extends smoothly over this exceptional orbit.

Finally, we will show that $u=\mathcal{O}(r^{-2})$ as $r\rightarrow\infty$ to ensure that it is in fact a K\"ahler conformal compactification.  First, on the $z_1-axis$ change coordinates to real polar coordinates by setting $z_1=(\sqrt{\z_1},\theta_1)$ where $\z_1=|z_1|^2=\mathcal{O}(\z)$.
Then, rewrite \eqref{omega_g}, the restriction of $\omega$ to the $z_1$-axis, in these coordinates
\begin{align}
\begin{split}
\omega&=2\Big(\frac{\partial \phi}{\partial \z}+\z\frac{\partial^2 \phi}{\partial \z^2}\Big)\sqrt{\z_1}(d\sqrt{\z_1}\wedge d\theta_1)+\sqrt{-1}\frac{\partial \phi}{\partial \z}\sum_{i=2}^ndz_i\wedge d\bar{z}_i\\
&=\frac{\partial}{\partial \sqrt{\z_1}}\Big(\z\frac{\partial \phi}{\partial \z}\Big)(d\sqrt{\z_1}\wedge d\theta_1)+\sqrt{-1}\frac{\partial \phi}{\partial \z}\sum_{i=2}^ndz_i\wedge d\bar{z}_i.
\end{split}
\end{align}
Since $\omega=g(\mathcal{J}_z\cdot,\cdot)$ and is ALE of order $\tau$, it must be asymptotic to the K\"ahler form of the standard Hermitian metric on $\CC^n$.  Examining the first term of $\omega$ restricted to the $z_1$-axis we see that as $r\rightarrow\infty$
\begin{align}
d\Big(\z\frac{\partial \phi}{\partial \z}d\theta_1\Big)=\frac{1}{2}d(r_1^2d\theta_1)+\mathcal{O}(r^{-\tau}),
\end{align}
where $r_1$ and $r$ denote the radial distances in the standard Hermitian metric along the $z_1$-axis and on all of $\CC^n$ respectively.
Therefore
\begin{align}
u=\Big(\z\frac{\partial \phi}{\partial\z}\Big)^{-1}=\mathcal{O}(r^{-2})
\end{align}
as $r\rightarrow\infty$, which completes the proof.

\subsection{Proof of Corollary \ref{potentialcorollary}}
By substituting formula \eqref{conformalfactor} for the conformal factor into the second equation of \eqref{equations}, we see that
\begin{align}
\label{exponentialpolar}
\frac{\partial \widehat{\phi}}{\partial \z}=-\Big(\z^2\frac{\partial \phi}{\partial\z}\Big)^{-1},
\end{align}
so therefore
\begin{align}
\widehat{\phi}(\z)=\int\frac{\partial \widehat{\phi}}{\partial \z}d\z=-\int\Big(\z^2\frac{\partial \phi}{\partial\z}\Big)^{-1}d\z.
\end{align}

\section{LeBrun ${\rm{U}}(2)$-invariant K\"ahler ALE metrics}
\label{LeBrunmetrics}
The remainder of this paper is dedicated to applications of Theorem~\ref{maintheorem}.  We will focus on a family of K\"ahler scalar-flat ALE spaces, due to LeBrun.  These arise in two seemingly distinct ways.  One way is by developing and solving a nonlinear ODE for a potential function, and the other is by using a hyperbolic analogue of the Gibbons-Hawking ansatz.  For our purposes, it will be simpler to focus on the former of the two constructions of which we give a brief description in Section \ref{negativemass}.

These spaces are real four-dimensional, and there are certain geometric properties, unique to this dimension, which will come into play.  Over an oriented Riemannian four-manifold, the Hodge star operator restricted to two-forms $*:\Lambda^2\rightarrow \Lambda^2$ satisfies $*^2=Id$.  This induces the decomposition of two-forms into the $\pm 1$ eigenspaces of $*|_{\Lambda^2}$.  The Weyl tensor has a corresponding decomposition $W=W^+\oplus W^-$, into the self-dual and anti-self-dual parts of the Weyl tensor respectively.  The metric is called anti-self-dual if $W^+=0$ and self-dual if $W^-=0$.  Note that by reversing orientation a self-dual manifold is converted into an anti-self-dual manifold and vice versa.  For a K\"ahler metric in real four-dimensions, $W^+$ is determined by the scalar curvature, so a K\"ahler scalar-flat metric is necessarily anti-self-dual.

\subsection{LeBrun metrics}
\label{negativemass}
In $1988$ LeBrun explicitly constructed a nonsingular anti-self-dual K\"ahler scalar-flat  ALE metric on the total space of the bundle $\mathcal{O}_{\PP^1}(-k)$ for all integers $k\geq 1$, see \cite{LeBrunnegative}.  For $k=1$ and $k=2$, these are the well-known Burns and Eguchi-Hanson metrics respectively \cite{Burns, EguchiHanson}.  When $k>2$, LeBrun showed that these metrics have negative mass thereby providing an infinite family of counter examples to the generalized positive action conjecture \cite{HawkingPope}, hence they are known as the {\em LeBrun negative mass metrics\em}.  These metrics can also be reworked in a way as to give a $1$-parameter family of anti-self-dual K\"ahler scalar-flat  ALE edge-cone metrics on $\mathcal{O}_{\PP^1}(-1)$, where the singular set is the $\CP^1$ at the origin, with cone angles $2\pi\beta$ for all $\beta>0$.  These will be referred to as the {\em LeBrun edge-cone metrics\em}.
We briefly describe LeBrun's method of construction to make it clear that Theorem \ref{maintheorem} is applicable.
For a more thorough description see \cite{LeBrunnegative, LeBrunJDG, LeBrunNN, ViaclovskyFourier, AtiyahLeBrun}.

Consider a K\"ahler potential $\phi(\z)$ on $\CC^2\setminus\{0\}$,
where $\z=|z_1|^2+|z_2|^2$.  The $(1,1)$-form
\begin{align}
\omega=\sqrt{-1}\partial\bar{\partial}\phi,
\end{align}
is the K\"ahler form of a metric on a spherical shell centered at the origin.  Recall that any 
K\"ahler metric satisfies
\begin{align}
\frac{R}{4}\omega\wedge\omega=\rho\wedge\omega,
\end{align}
where $R$ is the scalar curvature and $\rho$ the Ricci form.  Since LeBrun searched for K\"ahler scalar-flat metrics, he used the rotational symmetry of the situation to solve
\begin{align}
0=\rho\wedge\omega,
\end{align}
and thereby obtained a family of potential functions $\{\phi_{\beta}(\z)\}_{\beta\in\RR^+}$.  In fact, he obtained a wider family of potential functions, but for our purposes and in the interest clarity we restrict our attention to these.
For each $\phi_{\beta}(\z)$,  LeBrun defined a new radial coordinate 
 \begin{align}
 \label{r-coordinate}
 r=\sqrt{\z\frac{\partial\phi_{\beta}}{\partial \z}},
 \end{align}
and showed that the corresponding metric is
\begin{align}
\label{negativemassmetric}
g_{LB(\beta)}=\frac{dr^2}{1+\frac{\beta-2}{r^2}+\frac{1-\beta}{r^4}}+r^2\Big[\sigma_1^2+\sigma_2^2+\Big(1+\frac{\beta-2}{r^2}+\frac{1-\beta}{r^4}\Big)\sigma_3^2\Big],
\end{align}
where $r$ is the radial distance from the origin and $\sigma_1,\sigma_2, \sigma_3$ are the usual left-invariant coframe on ${\rm{SU}}(2)=S^3$.  Since for a K\"ahler metric in four-dimensions the self-dual part of the Weyl curvature tensor is determined by the scalar curvature, and these metrics are K\"ahler scalar-flat, they are anti-self-dual.

This metric is clearly singular at the origin.  However, by redefining the radial coordinate as $\tilde{r}^2~=~\beta^{-1}(r^2-1)$ and attaching a $\CP^1$ at $\tilde{r}=0$, one sees that
\begin{align}
\begin{split}
g_{LB(\beta)}&=d\tilde{r}^2+(\sigma_1^2+\sigma_2^2)+\beta^2\tilde{r}^2\sigma_3^2\\
&\phantom{=====}+\tilde{r}^2\Big[\frac{(\beta-1)}{\tilde{r}^2+1}d\tilde{r}^2+\beta(\sigma_1^2+\sigma_2^2)+\Big(\frac{1-\beta}{\beta\tilde{r}^2+1}\Big)\beta^2\tilde{r}^2\sigma_3^2\Big],
\end{split}
\end{align}
is in fact a K\"ahler scalar-flat ALE edge-cone metric on $(\mathcal{O}_{\PP^1}(-1),\CP^1)$ with cone angle $2\pi\beta$, where the singular set  is the $\CP^1$ at the origin.  Moreover, due to the construction, it is clearly ${\rm{U}}(2)$-invariant.  These are the {\em LeBrun egde-cone metrics\em}.  Note that when $\beta=1$ this is the Burns metric.

When $\beta=k$ is a positive integer, LeBrun obtained a nonsingular ALE metric.  
Consider the metric $g_{LB(k)}$ for
any positive integer $k$.  By taking the $\mathbb{Z}_k$ quotient of $\CC^2\setminus \{0\}$ generated by
\begin{align}
\label{Znaction}
(z_0,z_1)\mapsto(e^{2\pi i/k}z_0,e^{2\pi i/k}z_1),
\end{align}
which is rotation in the fiber,
it is clear that this metric extends smoothly over the $\CP^1$ at $\tilde{r}=0$.  Therefore $g_{LB(k)}$ defines a nonsingular
${\rm{U}}(2)$-invariant K\"ahler scalar-flat ALE  metric on the total space of $\mathcal{O}_{\PP^1}(-k)$.  These are the {\em LeBrun negative mass metrics\em}. It is important to distinguish that only when $\beta$ is a positive integer can we obtain a nonsingular metric.

\section{Explicit K\"ahler conformal compactifications}
\label{compactifications}

Here we use Theorem \ref{maintheorem} to examine K\"ahler conformal compactifications of the LeBrun ${\rm{U}}(2)$-invariant K\"ahler ALE spaces introduced in Section \ref{LeBrunmetrics}.  Since these spaces are K\"ahler scalar-flat they are necessarily anti-self-dual and  will therefore compactify to $C^{\infty}$-orbifold metrics, recall  \cite{TV2, CLW}.  In Section \ref{weighted} we provide a new and simple construction of the canonical Bochner-K\"ahler metric on $\CP^2_{(1,1,k)}$ for any positive integer $n$.  Then, in Section \ref{edgecone}, we explicitly construct a family of extremal K\"ahler edge-cone metrics on $(\CP^2,\CP^1)$ having cone angles $2\pi\beta$ for all $\beta>0$.  Also, in Remark \ref{Green's}, we note a interesting relationship between the conformal compactification factor and an explicit formulation of the Green's function for the conformal Laplacian on an orbifold.

\subsection{Bochner-K\"ahler metrics on weighted projective space}
\label{weighted}
Given relatively prime integer weights $1\leq p_0\leq p_1\leq\cdots\leq p_n$, the complex $n$-dimensional weighted projective space $\CP^n_{(p_0,p_1,\cdots,p_n)}$ is the quotient $S^{2n+1}/S^1$, where $S^1$ acts by
\begin{align}
(z_0,z_1,\cdots, z_n)\mapsto (e^{ip_0\theta}z_0,e^{ip_1\theta}z_1 ,\cdots, e^{ip_n\theta}z_n),
\end{align}
for $0\leq \theta <2\pi$.  This has the structure of a compact complex orbifold with the number of singular points corresponding to the number of weights greater than $1$.

Bryant proved that every weighted projective space admits a Bochner-K\"ahler metric \cite{Bryant}.   When $p_0=\cdots=p_n=1$, this metric is just the Fubini-Study metric on $\CP^n$.  Later, David and Gauduchon gave a direct construction of 
these metrics and used an argument due to Apostolov to show that this metric is the unique Bochner-K\"ahler metric on a given weighted 
projective space \cite{DavidGauduchon}.  Therefore, we refer to it as the {\em{canonical Bochner-K\"ahler metric}}.  See also \cite{Galicki-Lawson} for earlier related work. 
We do not discuss the construction of \cite{DavidGauduchon} here, because it involves sophisticated techniques in complex geometry, and the goal here is to focus on real $4$-dimensions and provide a simple construction for the canonical Bochner-K\"ahler metric on $\CP^2_{(1,1,k)}$ for any positive integer $k$.  In real $4$-dimensions, the Bochner tensor is exactly the anti-self-dual part of the Weyl tensor, so Bochner-K\"ahler metrics are the same as self-dual K\"ahler metrics.  It is interesting to remark that these metrics are in fact extremal K\"ahler, see \cite{Derdzinski}. 

Topologically $\CP^2_{(1,1,k)}=\widehat{\mathcal{O}}_{\PP^1}(-k)$.  Joyce proved that there is a quaternionic metric on $\CP^2_{(1,1,k)}$ which is conformal to the LeBrun negative mass metric on $\mathcal{O}_{\PP^1}(-k)$, see \cite{Joyce1991}.  This metric is necessarily the canonical Bochner-K\"ahler metric, however he does not find an explicit conformal factor or construction of the metric.  Here, using Theorem \ref{maintheorem}, we do just that.

\begin{theorem}
\label{wghtprojthm}
The canonical Bochner-K\"ahler metric on $\CP^2_{(1,1,k)}$ is a K\"ahler conformal compactification of the LeBrun negative mass metric on $\mathcal{O}_{\PP^1}(-k)$, and is explicitly given by
\begin{align*}
\label{BKmetric}
g_{BK(k)}=\frac{dr^2}{(r^2+k-1)(r^2-1)}+\frac{1}{r^2}\Big[\sigma_1^2+\sigma_2^2+\Big(1+\frac{k-2}{r^2}+\frac{1-k}{r^4}\Big)\sigma_3^2\Big].
\end{align*}
\end{theorem}

\begin{proof}
Recall the LeBrun negative mass metric $g_{LB(k)}$ on $\mathcal{O}_{\PP^1}(-k)$ from Section \ref{negativemass}.  This anti-self-dual K\"ahler ALE metric arises from the ${\rm{U}}(2)$-invariant K\"ahler potential $\phi_k(\z)$ on 
$\CC^2\setminus\{0\}$, and thus satisfies the conditions of Theorem \ref{maintheorem}.  Therefore
\begin{align}
 \widehat{g}_{LB(k)}=\Big(\z \frac{\partial \phi_k}{\partial \z}\Big)^{-2}g_{LB(k)}
 \end{align}
is a self-dual K\"ahler orbifold metric on $\widehat{\mathcal{O}}(-k)=\CP^2_{(1,1,k)}$.  Since the canonical Bochner-K\"ahler metric is the only such metric on $\CP^2_{(1,1,k)}$, it must be $\widehat{g}_{LB(k)}$.  
Now, recalling the coordinate change \eqref{r-coordinate} for the radial variable $r$, we see that the conformal factor $\big(\z \frac{\partial \phi_k}{\partial \z}\big)^{-2}$ is exactly $r^{-4}$ hence
\begin{align}
g_{BK(k)}=\widehat{g}_{LB(k)}=r^{-4}g_{LB(k)}.
\end{align}
\end{proof}

\begin{remark}
\label{Green's}
\em{Notice that $u^{-1}=r^2$ is the Green's function, associated to the orbifold point, for the conformal Laplacian on $(\CP^2_{(1,1,k)},g_{BK(k)})$.  In general, such a Green's function on a compact Riemannian orbifold is only guaranteed to exist if the Yamabe constant is nonnegative, and even when one is known to exist it is rare to know it explicitly.  It is useful when these functions exist because, given a compact Riemannian orbifold $(M,g)$ with a Green's function $G$ for the conformal Laplacian associated to a point $p\in M$, one can obtain a scalar-flat ALE space as the ``conformal blow-up''
\begin{align*}
(M\setminus \{p\}, G^2 g).
\end{align*}
Here a coordinate system at infinity arises from inverted normal coordinates around~$\{p\}$.
}
\end{remark}

\subsection{Extremal K\"ahler edge-cone metrics on $(\CP^2,\CP^1)$ }
\label{edgecone}

Extremal K\"ahler metrics were first introduced by Calabi in an effort to obtain a canonical metric in a given K\"ahler class \cite{Calabi1,Calabi2}.  Given a compact complex manifold with a K\"ahler metric, Calabi fixed the deRahm cohomology class of this metric and then considered the functional of the $L^2$-norm squared of the scalar curvature on the set of K\"ahler forms in this class.  Extremal K\"ahler metrics are the critical points on this functional.  Calabi showed that a K\"ahler metric is extremal K\"ahler if and only if the gradient of its scalar curvature is the real part of a holomorphic vector field, so in a way these metrics can be viewed as a generalization of constant scalar curvature K\"ahler metrics. 

Abreu constructed a $1$-parameter family of ${\rm{U}}(2)$-invariant extremal K\"ahler edge-cone metrics on $(\CP^2,\CP^1)$ having cone angles $2\pi\beta$ for all $\beta>0$, see \cite{Abreu}.  They are also self-dual.  Using the work of Derdzinski \cite{Derdzinski}, which we discuss in more detail below, Abreu then showed that, when the cone angle is restricted to $0<\beta<2$, these metrics are conformal to Einstein edge-cone metrics on $(\CP^2,\CP^1)$.  In \cite{AtiyahLeBrun}, Atiyah-LeBrun gave an independent construction of these Einstein metrics.  They  begin by using a hyperbolic analogue of the Gibbons-Hawking ansatz to construct metrics on the total space of a ${\rm{U}}(1)$-bundle over hyperbolic $3$-space minus a point.  (These metrics are conformal to the metrics $g_{LB(\beta)}$.)  Then, they used results from \cite{LeBrunNN, HitchinJDG} to solve for the explicit conformal factor that yields these Einstein edge-cone metrics.  The restriction of $0<\beta<2$ on the cone angle here is necessary to obtain a metric.

Abreu, in fact, proved much more general results, which we will not discuss here.  In this particular case however, we are actually able to explicitly construct the $1$-parameter family of ${\rm{U}}(2)$-invariant extremal K\"ahler edge-cone metrics on $(\CP^2,\CP^1)$, having cone angles $2\pi\beta$ for all $\beta>0$, in a very simple way by applying Theorem \ref{maintheorem} to the LeBrun edge-cone metrics discussed in Section \ref{negativemass}.  It is almost exactly how we obtained the canonical Bochner-K\"ahler metrics on $\CP^2_{(1,1,k)}$ above.

Recall the family of LeBrun edge-cone metrics $g_{LB(\beta)}$ on $(\mathcal{O}_{\PP^1}(-1),\CP^1)$, with cone angles $2\pi\beta$ for all $\beta>0$, from Section \ref{negativemass}.  By Theorem \ref{maintheorem}, for all $\beta>0$ we have the compactification 
\begin{align}
\widehat{g}_{LB(\beta)}=\Big(\z \frac{\partial \phi_{\beta} }{\partial \z}\Big)^{-2}g_{LB(\beta)}
\end{align}
to a self-dual K\"ahler edge-cone metric on $(\CP^2,\CP^1)$ with cone angle $2\pi\beta$.
Changing coordinates to the radial variable $r$ as in \eqref{r-coordinate}, we see that $\widehat{g}_{LB(\beta)}=r^{-4}g_{LB(\beta)}$ so on $(\CP^2,\CP^1)$ this metric is explicitly
\begin{align}
\widehat{g}_{LB(\beta)}=\frac{dr^2}{(r^2+\beta-1)(r^2-1)}+\frac{1}{r^2}\Big[\sigma_1^2+\sigma_2^2+\Big(1+\frac{\beta-2}{r^2}+\frac{1-\beta}{r^4}\Big)\sigma_3^2\Big].
\end{align}
In real four dimensions, Derdzinski showed that the gradient of the scalar curvature of self-dual K\"ahler metric is the real part of a holomorphic vector field \cite{Derdzinski}, hence the metric itself is necessarily extremal K\"ahler.  Therefore, $\widehat{g}_{LB(\beta)}$ is extremal K\"ahler, and thus we have completed the construction of the desired family ${\rm{U}}(2)$-invariant extremal K\"ahler edge-cone metrics.

It is important to distinguish that in Section \ref{weighted} we began with the nonsingular ALE metric $g_{LB(k)}$ on $\mathcal{O}_{\PP^1}(-k)$ and compactified to obtain an orbifold with isolated singularities, while here we started with the edge-cone ALE metric $g_{LB(\beta)}$ on $(\mathcal{O}_{\PP^1}(-1),\CP^1)$ and compactified to obtain a compact manifold with an edge-cone singularity.

Now, we examine the conformal relationship between the $\widehat{g}_{LB(\beta)}$ and (locally)  Einstein metrics.  The procedure discussed below was followed in \cite{Abreu} to obtain the Einstein edge-cone metrics mentioned earlier, however we expand upon this to when the metric is not globally conformally Einstein, which is when $\beta\geq 2$.  From our formulation, the reader is able to obtain a clear picture of the geometry. 
These results have been obtained previously so we only give a brief description here, see \cite{PedersenPoonKSF, HitchinJDG, CalderbankPedersen, AtiyahLeBrun} for more thorough discussions.
However, the previous results in this direction relied upon twistor theory or finding the solution to an equation involving a conformal change in Ricci curvature, while we will need only to compute the scalar curvatures of the aforementioned family of extremal K\"ahler metrics.

Derdzinski showed that a self-dual K\"ahler metric $g$ is conformal to a self-dual Hermitian Einstein metric on $M^*:=\{x\in M:R(x)\neq 0\}$, given by $\tilde{g}=R^{-2}g$, where $R$ is the scalar curvature \cite{Derdzinski}.  The conformal metric $\tilde{g}$ is no longer K\"ahler unless $R$ is constant.  
Therefore, we wish to compute the scalar curvature of the metrics $\widehat{g}_{LB(\beta)}$ on $(\CP^2,\CP^1)$, which we denote by $\widehat{R}$.  (We suppress the particular $\beta$ to which this is with respect because it will be clear from the context.)
Since $g_{LB(\beta)}$ is scalar flat, the well-known formula for the scalar curvature of a conformal metric reduces to the equation
\begin{align}
\label{conformalscalar}
\widehat{R}=-\frac{6}{u^3}\Delta_{g_{LB(\beta)}}u,
\end{align}
where $u=\big(\z \frac{\partial \phi_{\beta} }{\partial \z}\big)^{-1}=r^{-2}$, so we can compute that
\begin{align}
\widehat{R}=24\Big[(2-\beta)+\frac{2(\beta-1)}{r^2}\Big].
\end{align}
To see where this metric is conformally Einstein, we examine if and when
$\widehat{R}=0$.  Away from where the scalar curvature vanishes, these metrics will be conformally Einstein with Einstein constant $6\beta^2(2-\beta)$.
We have the following cases:
\begin{enumerate}
\item 
When $0<\beta<2$, the scalar curvature $\widehat{R}$ is everywhere positive.
Therefore, the metric $\widehat{R}^{-2}\widehat{g}_{LB(\beta)}$
is a self-dual Einstein edge-cone metric, with positive Einstein constant, on $(\CP^2,\CP^1)$ with cone angle $2\pi\beta$.  It is easy to check that, after scaling, this construction gives the Einstein metrics found independently by Atiyah-Lebrun and Abreu discussed earlier.  Only for this range of cone angles does the scalar curvature not vanish somewhere, hence they are the only globally conformally Einstein metrics.

\item 
When $\beta=2$, the scalar curvature 
$\widehat{R}$ vanishes at the point of compactification, so the metric is conformally Einstein by a factor of $\widehat{R}^{-2}=r^4/48^2$, with vanishing Einstein constant, away from this point.  Notice that the conformal factor is a scalar multiple of the inverse of the K\"ahler conformal compactification factor, so the conformal metric will be a scaled version of the original ALE metric.  This is the Eguchi-Hanson metric on
$\mathcal{O}_{\PP^1}(-1)$, which is the double cover branched over the $\CP^1$ at the origin of the usual space \cite{EguchiHanson}.

\item 
When $\beta>2$, the scalar curvature $\widehat{R}$ vanishes along the hypersurface defined by $r^2=2(\beta-1)/(\beta-2)$.  The compliment is composed of two pieces on each of which the metric is conformal to a self-dual asymptotically hyperbolic Einstein (AHE) metric by a factor of $\widehat{R}^{-2}$.  The piece containing the singular set $\CP^1$ is diffeomorphic to $\mathcal{O}_{\PP^1}(-1)$, and on this we obtain a family of self-dual AHE edge-cone metrics.  Observe that, when $\beta=k$ is a positive integer, the corresponding AHE edge-cone metric has a quotient as in \eqref{Znaction} to a nonsingular self-dual (locally) AHE metric on $\mathcal{O}_{\PP^1}(-k)$ with boundary a lens space.  These are the well-known Pedersen-LeBrun metrics, see \cite{HitchinJDG,CalderbankSinger}.  
The piece containing the point of compactification is diffeomorphic to the $4$-ball, and on this we obtain a family of smooth self-dual AHE metrics.  After changing variables and rescaling, we see that these are in fact the well-known Pedersen metrics \cite{Pedersenmetrics}.
\end{enumerate}

\bibliography{conformal_factor_references}

\end{document}